\documentclass[11pt]{article}
\usepackage{amsmath, amssymb, amsthm}
\usepackage{verbatim}
\usepackage{multicol}

\date{}
\title{\vspace{-0.8cm}A random triadic process }
\author{
D\'aniel Kor\'andi \thanks{Department of Mathematics, ETH, 8092 Zurich. Email: daniel.korandi@math.ethz.ch.}
\and
Yuval Peled \thanks{School of Computer Science and Engineering, The Hebrew University of Jerusalem, Jerusalem, Israel. Email: yuvalp@cs.huji.ac.il. Yuval Peled is grateful to the Azrieli Foundation for the award of an Azrieli Fellowship.}
\and
Benny Sudakov \thanks{Department of Mathematics, ETH, 8092 Zurich.
Email: benjamin.sudakov@math.ethz.ch. 
Research supported in part by SNSF grant 200021-149111.}
}

\theoremstyle{plain}
\newtheorem{THM}{Theorem}[section]
\newtheorem*{THM*}{Theorem}
\newtheorem{PROP}[THM]{Proposition}
\newtheorem{LEMMA}[THM]{Lemma}

\newtheorem{COR}[THM]{Corollary}
\newtheorem{CLAIM}[THM]{Claim}

\theoremstyle{definition}

\newcommand{\Prb}{\mathbf{P}}
\newcommand{\Exp}{\mathbf{E}}

\newcommand{\subs}{\subseteq}
\newcommand{\eps}{\varepsilon}
\newcommand{\Bin}{\textrm{Bin}}

\newcommand{\mB}{\mathcal{B}}
\newcommand{\mG}{\mathcal{G}}

\newcommand{\polylog}{\textrm{polylog }}

\begin{document}
\maketitle

\begin{abstract}
Given a random 3-uniform hypergraph $H=H(n,p)$ on $n$ vertices where each triple independently appears with probability $p$, consider the following {\em graph} process. We start with the star $G_0$ on the same vertex set, containing all the edges incident to some vertex $v_0$, and repeatedly add an edge $xy$ if there is a vertex $z$ such that $xz$ and $zy$ are already in the graph and $xzy \in H$. We say that the process {\em propagates} if it reaches the complete graph before it terminates. In this paper we prove that the threshold probability for propagation is $p=\frac{1}{2\sqrt{n}}$. We conclude that $p=\frac{1}{2\sqrt{n}}$ is an upper bound for the threshold probability that a random 2-dimensional simplicial complex is simply connected.
\end{abstract}

\section{Introduction}

The principle of {\em triadic closure} is an important concept in social network theory (see e.g. \cite{EKBOOK}). Roughly speaking, it says that when new friendships are formed in a social network, it is more likely to occur between two people sharing a common friend, thus ``closing'' a triangle, than elsewhere. We will consider a simplistic model of the evolution of a social network, where friendships can only be formed through a common friend, and triadic closure eventually occurs at any triangle with probability $p$, independently of other triangles. We refer to this process as the {\em triadic process}. 

Formally, let $H=H(n,p)$ be a random 3-uniform hypergraph on $[n]$ where each triple independently appears with probability $p$. The triadic process is the following {\em graph} process. We start with the star $G_0$ on the same vertex set $[n]$, containing all the edges incident to some vertex $v_0$, and repeatedly add any edge $xy$ if there is a vertex $z$ such that $xz$ and $zy$ are already in the graph and $xzy \in H$. We say that the process {\em propagates} if all the edges are added to the graph eventually. It is easy to see that this event does not depend on the order the edges are added in. In this paper we prove that the threshold probability for propagation is $\frac{1}{2\sqrt{n}}$.

\begin{THM}\label{thm:main}
Suppose $p=\frac{c}{\sqrt{n}}$, for some constant $c>0$. Then,
\begin{enumerate}
\item If $c>\frac 12$, then the triadic process propagates whp.
\item If $c<\frac 12$, then the triadic process stops at $O(n\sqrt{n})$ edges whp.
\end{enumerate}
\end{THM}

As usual, we say that some property holds {\em with high probability} or {\em whp} if it holds with probability tending to 1 as $n$ tends to infinity.

\medskip
Randomized graph processes have been intensively studied in the past decades. One notable example is the triangle-free process, originally motivated by the study of the Ramsey number $R(3,n)$ (see e.g. \cite{ESW95}). In this process the edges are added one by one at random as long as they do not create a triangle in the graph. The triadic process is a slight variant of this, with a very similar nature. Indeed, our analysis makes good use of the tools developed by Bohman \cite{B09} when he applied the differential equation method to track the triangle-free process. Several other related processes were also analyzed using differential equations, e.g. \cite{BFL14}. For more information about this method we refer the interested reader to the excellent survey of Wormald \cite{W99}.

Coja-Oghlan, Onsj\"o and Watanabe \cite{COW12} investigated a similar kind of closure while analyzing connectivity properties of random hypergraphs. They say that a 3-hypergraph is propagation connected if its vertices can be ordered in some way  $v_1,\ldots,v_n$ so that each $v_i$ ($i\ge 3$) forms a hyperedge with two preceding vertices. They obtain the threshold probability for the propagation connectivity of $H(n,p)$ up to a small multiplicative constant. Using this directed notion of connectivity, our problem asks when the random 3-hypergraph on the {\em line graph} of $K_n$ is propagation connected from the star.

\medskip

Our main motivation for considering the triadic process comes from the theory of random 2-dimensional simplicial complexes. A simplicial 2-complex on the vertex set $V$ is a set family $Y \subs \binom{V}{\le 3}$ closed under taking subsets. The dimension of a simplex $\sigma\in Y$ is defined to be $|\sigma|-1$. We use the terms vertices, edges and faces for 0, 1 and 2-dimensional simplices, respectively. The 1-skeleton of a 2-complex is the subcomplex containing its vertices and edges.

The Linial--Meshulam model of random simplicial complexes, introduced in \cite{LM06}, is a generalization of the Erd\H{o}s--R\'enyi random graph model and has been studied extensively in recent years. The random 2-complex  $Y_2(n,p)$ is defined to have the complete 1-skeleton, i.e., all vertices and edges, and each of the faces independently with probability $p$. The study of random complexes involves both topological invariants and combinatorial properties, including homology groups, homotopy groups, collapsibility, embeddability and spectral properties.

One of the oldest questions of this kind, asked by Linial and Meshulam \cite{LM06}, is how the fundamental group $\pi_1(Y_2(n,p))$ of the random 2-complex behaves. Babson, Hoffman and Kahle \cite{BHK11} showed that if $p<n^{-\alpha}$ for some arbitrary $\alpha>1/2$ then the fundamental group is nontrivial whp. On the other hand, they proved that $\pi_1(Y_2(n,p))$ is trivial for $p>\sqrt{4\log n/n}$, which means that the threshold probability for being simply connected should be close to $n^{-1/2}$. As a corollary of the first part of Theorem~\ref{thm:main}, we improve the upper bound on the threshold probability.

\begin{COR} \label{cor:top}
Let $p=\frac{c}{\sqrt{n}}$ for some constant $c>\frac{1}{2}$. Then $Y_2(n,p)$ is simply connected whp.
\end{COR}
\begin{proof}
Suppose we have a 2-complex $C$ such that one of its edges, $e$, is contained in a unique face $f$. Then we can {\em collapse} $f$ onto the other two edges without changing the fundamental group of $C$. In fact, $C-f-e$ is homotopy equivalent to $C$, the former complex being a deformation retract of the latter. We say that a 2-complex with complete 1-skeleton is a {\em collapsible hypertree} if we can apply a sequence of collapses to it and end up with a tree. Clearly, a collapsible hypertree has trivial fundamental group.

Now observe that Theorem~\ref{thm:main} implies that $Y_2(n,p)$ contains a collapsible hypertree whp. Indeed, if the process propagates, then take $C$ to be the subcomplex of the faces that correspond to the triples we used to add edges to the graph. Then by definition, the reverse of the triadic process on $C$ is exactly a sequence of collapses resulting in a star.

Basic results about the topology of complexes tell us that the addition of faces to a simply connected complex does not change the fundamental group, hence $\pi_1(Y_2(n,p))$ is trivial whp.
\end{proof}

\subsection{Proof outline}

Instead of exposing all the triples at once, we will be sampling them on the fly, trying to extend the edge set of the graph. Both the proofs of the upper bound and the lower bound consist of two phases. 
In the first phase we make one step at a time: we choose, uniformly at random, one (yet unsampled) triple spanning exactly two edges and expose it. With probability $p$ the triple is selected, hence we can add the third edge to our edge set.
The second phase proceeds in rounds: we simultaneously expose all the unsampled triples spanning two edges, and extend the edge set according to the outcome.

The essence of the proof is to track the behavior of certain variables throughout the process. As we will see, this is not a very hard task to do in the second phase, using standard measure concentration inequalities. However, during the initial phase of the process, the codegrees (one of the variables we track) are not concentrated, which forces us to do a more careful analysis of the beginning of the process. For this we will use the differential equation method.
\medskip

We organize the rest of the paper as follows. In Section~\ref{sec:diff} we give an overview of how we apply the differential equation method. A detailed analysis of the actual implementation follows in Section~\ref{sec:calc}. We move on to the second phase of the process in Section~\ref{sec:phasetwo}, thereby completing the proof of Theorem~\ref{thm:main}. We finish the paper with some further remarks in Section~\ref{sec:last}.

Notations: Throughout the paper, we will omit floor and ceiling signs whenever they are not necessary. The sign $\pm$ will be used to represent both a two-element set of values and a whole interval, but it should be clear from the context which one is the case.

\section{The differential equation method} \label{sec:diff}

At any point in the process, we say that a vertex triple $\{u,v,w\}$ is {\em open} if it spans exactly two edges but has not yet been sampled. We will also use the notation $uvw$ for an open triple with edges $uv$ and $vw$. By an {\em open triple at $u$}, we mean a triple $uvw$, i.e., one that has its missing edge adjacent to the vertex $u$.

In each step, our process picks an open triple uniformly at random and samples it. If the answer is positive then we close the triple by adding the missing edge to the graph.
To analyze this process we apply the differential equation method, using some ideas from \cite{B09}. 

\medskip

For simplicity, let us denote the graph we obtain after $i$ samples by $G_i$. We consider the following random variables: $D_v(i)$ is the degree of the vertex $v$ in $G_i$. $F_v(i)$ is the number of open triples at $v$, so it is the number of ways for $v$ to gain a new incident edge in $G_{i+1}$. $X_{u,v}(i)$ is the codegree of $u$ and $v$, i.e., the number of common neighbors of $u$ and $v$ in $G_i$.

To provide some insight, we first heuristically describe the process. Let us assume for now that the $D_v(i)$ are concentrated around some value $D(i)$, and similarly the $F_v(i)$ are approximately equal to some value $F(i)$. We further assume that the variables are very close to their expectations.

In step $i+1$ we choose an open triple uniformly at random, so each triple is chosen with probability $\frac{2}{\sum_v F_v(i)}\approx \frac{2}{nF(i)}$, and then sample it. With probability $p$ the sample is successful, hence we can close the triple. As the number of open triples at a vertex $v$ is about $F(i)$, the change in the degree of a vertex $v$ we expect to see is
\[ D(i+1)-D(i)\approx \frac{2p}{n}. \]

Now let us see how $F_v(i)$ is affected by a step. We gain open triples at $v$ either if we successfully sample one of them (adding the edge $vw$), in which case new open triples are formed with the neighbors of $w$, or if we successfully sample a triple at some neighbor of $v$. On the other hand, we lose the sampled triple regardless of the outcome. The probability of sampling an open triple at some specific vertex $w$ is $\frac{2F_w(i)}{\sum_v F_v(i)}\approx \frac{2}{n}$, so assuming all the codegrees are negligible compared to $D(i)$, the expected change is 
\[ F(i+1)-F(i)\approx \frac{2}{n}(2pD(i)-1). \]

To smooth out this discrete process, we introduce a continuous variable $t$ and say that step $i$ corresponds to time $t=t_i=\frac{i}{n^2}$. 
Let us also rescale $D$ and $F$ by considering the smooth functions $d$ and $f$ in $t$, where we want $d(t)$ to be approximately $D(i)/\sqrt{n}$ and $f(t)$ to be approximately $F(i)/n$. Note that, since $p=c/\sqrt{n}$, our assumptions so far suggest the following behavior:
\[ d'(t) \approx \frac{d(t+1/n^2)-d(t)}{1/n^2} \approx n^{3/2}(D(i+1)-D(i)) \approx 2c \]
and
\[ f'(t) \approx \frac{f(t+1/n^2)-f(t)}{1/n^2} \approx n(F(i+1)-F(i)) \approx 4cd(t)-2.\]

Let us emphasize that this little musing that we are presenting here is not a proof at all --- a detailed analysis and the proof of concentration will follow in Section~\ref{sec:calc}. However, it at least indicates why it is plausible to believe that the actual values of $D_v(i)$ and $F_v(i)$ follow the trajectories of $d$ and $f$ given by the system of differential equations $d'(t)=2c$ and $f'(t)=4cd(t)-2$.

\bigskip

In the previous paragraphs we made the assumption that the codegrees are negligible compared to the degrees, but since they are not concentrated, proving this still needs some thought. To this end, we introduce two more random variables. $Y_{u,v}(i)$ denotes the number of {\em open 3-walks} $uww'v$ from $u$ to $v$, i.e., 3-walks where we require that $uww'$ be open (but allowing $w=v$), and $Z_{u,v}(i)$ is the number of {\em open 4-walks} $uww'w''v$ (again, allowing vertex repetitions), where both $uww'$ and $w'w''v$ are open. Note that $Y_{u,v}$ is {\em not} symmetric in $u$ and $v$.

The point is that $Y_{u,v}$ and $Z_{u,v}$ are concentrated (as we will see in Section \ref{sec:calc}), and --- amazingly enough --- their one-step behavior can be described with fairly simple formulas. So let us continue with our thought experiment and assume that all $Y_{u,v}(i)\approx Y(i)$, all $Z_{u,v}(i)\approx Z(i)$, and all variables are close to their expectations.

First of all, the increase in the codegrees comes from a successful sample in a 3-walk, so we expect 
\[ X_{u,v}(i+1)-X_{u,v}(i) \approx \frac{2p(Y_{u,v}(i)+Y_{v,u}(i))}{nF(i)} \approx \frac{4cy(t)}{n^2 f(t)}. \]
This will be enough to prove a uniform $O(\log n)$ upper bound over all the codegrees, so we can keep ignoring the effect of $X$ in the next few paragraphs.

Let us look at the change in $Y_{u,v}(i)$. There are three different ways a new open 3-walk $uww'v$ can appear after step $i+1$, depending on which one of $uw,ww'$ and $w'v$ is the new edge. When $uw$ is the new edge, there is a 4-walk $utww'v$ in $G_i$ where $utw$ is open. We can count such configurations by first choosing $w'$ as a neighbor of $v$ and then choosing an open 3-walk $utww'$. Note that for any such choice, $uww'$ will be an open triple in $G_{i+1}$, except if $w'$ is the same as $u$, or a common neighbor of $u$ and $w$. The latter cases are negligible, so there are about $Y(i)D(i)$ possibilities in this case.

Similarly, when $ww'$ is the new edge, new 3-walks come from 4-walks $uwtw'v$, and we can count the number of options by first choosing $w$ as a neighbor of $u$ and then an open 3-walk from $w$ to $v$. Again, the triple $uww'$ will be open in $G_{i+1}$ if $w'$ is neither $u$, nor a common neighbor of $u$ and $v$, so we find $Y(i)D(i)$ possibilities of this type. Finally, $w'v$ can only be the new edge if $w'tv$ was successfully sampled in some open 4-walk $uww'tv$, so there are about $Z(i)$ such options.

On the other hand, we lose an open 3-walk if we sample its open triple, whether or not the sample is successful. As any particular triple is chosen with probability about $\frac{2}{nF(i)}$, this means that we expect to see
\[ Y(i+1)-Y(i) \approx \frac{2}{nF(i)} \Big( p\big(2Y(i)D(i)+Z(i)\big) - Y(i) \Big). \]

The change in $Z(i)$ is a bit easier to analyze: Once again, we obtain a new 4-walk $uww'w''v$ if one of its edges is added in step $i+1$. We will assume it is the first edge, $uw$, but by symmetry our counting argument works for all other edges, as well. Then the 4-walk comes from a 5-walk $utww'w''v$ in $G_i$. We can count the number of options by first taking an open triple $vw''w'$ at $v$ and then choosing an open 3-walk from $u$ to $w'$. Again, the created 4-walk will automatically be open unless $w'$ is $u$ or a neighbor of $u$, so there are about $Y(i)F(i)$ candidates of this type and $4Y(i)F(i)$ in total. And then of course, we lose an open 4-walk if we sample one of its two open triples, regardless of the outcome. This suggests
\[ Z(i+1)-Z(i) \approx \frac{2}{nF(i)} \Big( 4pY(i)F(i) - 2Z(i) \Big). \]

Once again, we are looking for smooth functions $y$ and $z$ such that $y(t)$ is approximately $Y(i)/\sqrt{n}$ and $z(t)$ is about $Z(i)/n$. Then the same computation as before gives the differential equations
\[ y'(t) = \frac{2}{f(t)}\big((2cd(t)-1)y(t)+cz(t)\big) \]
and
\[ z'(t) = \frac{4}{f(t)}\big(2cy(t)f(t)-z(t)\big). \]

\medskip

We have yet to talk about the initial conditions of the above system of differential equations. Our process starts with a star centered at some vertex $v_0$, i.e., an $n$-vertex graph with $n-1$ edges, all of them touching $v_0$. Then $D_v(0)=1$, $F_v(0)=n-2$, $Y_{u,v}(0)=0$ and $Z_{u,v}(0)=n-3$ for any two vertices $u$ and $v$ other than $v_0$. For convenience, we will drop the center of the star from consideration in the sense that we do not define the variables with $v_0$ among the indices. This is a technicality that allows us to prove concentration, and since our recurrence relations never use those variables, it causes no problem.

Hence we obtain the initial conditions $d(0)=0$, $f(0)=1$, $y(0)=0$ and $z(0)=1$, and an easy calculation shows that the corresponding solution of our system of differential equations is
\[ \begin{array}{c c c}
    d(t)=2ct & \quad \quad & f(t)=1-2t+4c^2t^2 \\
    y(t)=d(t)f(t) & \quad \quad & z(t)=f^2(t).
\end{array} \]
In the next section we prove that the variables indeed closely follow the paths defined by these functions.

\section{Calculations} \label{sec:calc} 

In this section we show that our variables follow the prescribed trajectories up to some time $T$. Of course, we cannot hope to do so if $f(t)$ vanishes somewhere on $[0,T]$, as that would mean that the process is expected to die before time $T$. Now if $c>1/2$ then $f$ has no positive root, so this is not an issue: we can take $T=\sqrt{\log n}$. However, if $c\le 1/2$ then $f$ does reach 0, first at time $T_0=\frac{1-\sqrt{1-4c^2}}{4c^2}$. In this case $T$ will be chosen to be a constant arbitrarily close to $T_0$.

The allowed deviation of each variable will be defined by one of the error functions
\[ g_1(t)=e^{Kt}n^{-1/6} \qquad \quad \mbox{and} \qquad \quad g_2(t)=(1+d(t))e^{Kt}n^{-1/6}, \]
where
\[K = 100 \cdot \max_{0\le t\le T} \left( 1+ \frac{d(t)}{f(t)} + \frac{1}{f(t)}\right). \]

It is clearly enough to prove the first part of Theorem~\ref{thm:main} for $c\le 1$, so from now on we will assume this is the case.

Let us define $\mG_i$ to be the event that all of the bounds below in Proposition~\ref{prop:main}(a)-(e) hold for every pair of vertices $u$ and $v$ and for all indices $j=0,\ldots,i$. This section is devoted to the proof of the following result, which is the key to proving that the variables follow the desired trajectories.

\begin{PROP}  \label{prop:main}
Fix some vertices $u$ and $v$. Then, conditioned on $\mG_{j-1}$, each of the following bounds fails with probability at most $n^{-10}$.
\begin{multicols}{2}
\begin{enumerate}
 \item[(a)] $D_v(j) \in \big(d(t_j)\pm g_1(t_j) \big)\sqrt{n}$
 \item[(b)] $F_v(j) \in \big(f(t_j)\pm g_1(t_j) \big)n$
 \item[(c)] $X_{u,v}(j) \le 50\log n$
 \item[(d)] $Y_{u,v}(j) \in \big(y(t_j)\pm g_2(t_j) \big)\sqrt{n}$
 \item[(e)] $Z_{u,v}(j) \in \big(z(t_j)\pm g_2(t_j) \big)n$
 \item[]
\end{enumerate}
\end{multicols}
\end{PROP}

As a corollary, we obtain our main result.

\begin{THM} \label{thm:diffeq}
Suppose $c\le 1$, and $T\le \sqrt{\log n}$ and $K$ are defined as above. Then the bounds in Proposition~\ref{prop:main}(a)-(e) hold with high probability for all vertices $u$ and $v$ and for every $j=0,\ldots,T\cdot n^2$.
\end{THM}

\begin{proof}
It is easy to check that $\mG_0$ always holds. If $\mB_j$ is the event that, conditioned on $\mG_{j-1}$, at least one of these bounds fails for $j$, then the failure probability is exactly $\Prb[\cup_{j=1}^{Tn^2}\mB_j]$. A trivial union bound over all pairs of vertices and all equations in Proposition~\ref{prop:main} shows that $\Prb[\mB_j]\le 5n^{-8}$, hence another union bound over the indices gives $\Prb[\cup_{j=1}^{Tn^2}\mB_j]\le n^{-5}=o(1)$.
\end{proof}

To prove Proposition~\ref{prop:main}, we follow the strategy in \cite{B09} and analyze each random variable separately. Our plan is to use some martingale concentration inequalities to bound the probability of large deviation. However, since we cannot track the exact values of the expectations, only estimate them by some intervals, we will use two separate sequences to bound each variable: A submartingale to bound from below, and a supermartingale to bound from above.

Recall that a stochastic process $X_0,X_1,\ldots$ is called a {\em submartingale} if $\Exp[X_{i+1}|X_1,\ldots,X_i]\ge X_i$ for all $i$, and a {\em supermartingale} if $\Exp[X_{i+1}|X_1,\ldots,X_i]\le X_i$ for all $i$. We say that a sequence $X_0,X_1,\ldots $ of variables is $(\eta,N)$-bounded if $X_i-\eta\le X_{i+1}\le X_i+N$ for all $i$. We call a sequence of pairs $X_0^{\pm},X_1^{\pm},\ldots$ an $(\eta,N)$-bounded martingale pair, if $X_0^+,X_1^+,\ldots$ is an $(\eta,N)$-bounded submartingale and $X_0^-,X_1^-,\ldots$ is an $(\eta,N)$-bounded supermartingale. The following concentration results of Bohman \cite{B09} are essential for proving that the variables follow the desired trajectories:

\begin{LEMMA}[Bohman] \label{lem:mart}
Suppose $\eta\le N/10$ and $a<\eta m$. If $0\equiv X_0^{\pm},X_1^{\pm},\ldots$ is an $(\eta,N)$-bounded martingale pair then
\[ \Prb[X_m^+\le -a] \le e^{-\frac{a^2}{3\eta mN}} \qquad \mbox{and} \qquad \Prb[X_m^-\ge a] \le e^{-\frac{a^2}{3\eta mN}}. \]
\end{LEMMA}

\medskip
The general idea for analyzing a random variable $R(i)$, representing any of the above five variables, is the following. In step $i$, an open triple is sampled, and thus with probability $p$ a new edge is added to our graph. We split the one-step change in $R(i)$ into two non-negative variables: $A_i$ is the gain and $C_i$ is the loss in step $i$, so $R(j)=R(0)+\sum_{i=1}^j A_i-C_i$. The gain comes from the contribution of the added edge after a successful sample. Loss can only occur when some open triple stops being open, either because it was sampled or because its missing edge was added through some other open triple (although the effect of the latter event is negligible compared to the former if the codegrees are small).

Next we estimate the expectation of $A_i$ (using the recurrence relations we hinted at in Section~\ref{sec:diff}), so that we can define $A_i^+$ and $A_i^-$, shifted copies of $A_i$ with non-negative and non-positive expectations, respectively. This way $B_j^{\pm}=\sum_{i=1}^j A_i^{\pm}$ is an $(\eta,N)$-bounded martingale pair, where $\eta$ is approximately the expectation and $N$ is some trivial upper bound on $A_i$. We do the same with the $C_i$ to define $C_i^{\pm}$ and the martingale pair $D_j^{\pm}= \sum_{i=1}^j C_i^{\pm}$.

Finally we establish a connection between the concentration of $R(j)=R(0)+\sum_{i=1}^j A_i-C_i$ and the concentration of our shifted variables $B_j^{\pm}$ and $D_j^{\pm}$ in Lemma~\ref{lem:bound}, and then use the concentration of martingale pairs, Lemma~\ref{lem:mart}, to bound the error probabilities in Corollary~\ref{cor:prob}.

\medskip
The rest of this section is devoted to the actual calculations. The reader might want to skip the details at a first reading.
The first subsection establishes the tools we use to prove concentration, while the remaining five subsections prove one-by-one the five parts of Proposition~\ref{prop:main}.

\subsection{Tools}

The following claim will help us clean up the calculations of the expectations. Recall that $K= 100 \cdot \max_{0\le t\le T} \left( 1+ \frac{d(t)}{f(t)} + \frac{1}{f(t)}\right)$. 

\begin{CLAIM} \label{lem:squeeze}
Let $0\le t\le T$ so that $f(t)>0$ is bounded away from 0 ($t$ might depend on $n$). If $r(t)$ is one of the functions 1, $d(t)$ or $f(t)$ then
\begin{align*}
  \frac{( r(t)\pm g_1(t) )( f(t)\pm g_1(t) ) \left( 1 + O(\frac{\log n}{\sqrt{n}}) \right) }{f(t)\pm g_1(t)} &\subs r(t) \pm \frac{K}{20}g_1(t)  &\qquad \mbox{and} \\
  \frac{( r(t)\pm g_1(t) )( y(t)\pm g_2(t) ) \left( 1 + O(\frac{\log^2 n}{\sqrt{n}}) \right) }{f(t)\pm g_1(t)} &\subs r(t)d(t) \pm \frac{K}{20}g_2(t)  &\qquad \mbox{and also} \\
  \frac{(z(t)\pm g_2(t)) \left( 1 + O(\frac{\log n}{\sqrt{n}}) \right) }{f(t)\pm g_1(t)} &\subs f(t) \pm \frac{K}{20}g_2(t)
\end{align*}
\end{CLAIM}
\begin{proof}
Straightforward calculus shows that
\[\frac{1}{f(t)\pm g_1(t)} \subs \left( \frac{1}{f(t)}\pm \frac{g_1(t)}{f^2(t)} + O\left(\frac{g_1^2(t)}{f^3(t)}\right)\right). \]
Using this, we will multiply out the formulas on the left-hand side of the inequalities. Note that $g_1(t)$ and $g_2(t)$ are both $O(n^{-1/7})$, so in the expanded formulas, any term containing two factors of the type $g_{\alpha}(t)$ or a factor of $O(\frac{\polylog n}{\sqrt{n}})$ is consumed by an $O(n^{-2/7})$ error term. Hence the left-hand side of the first inequality is contained in
\begin{align*}
  &~~~ (r(t)\pm g_1(t) )( f(t)\pm g_1(t) ) \left( 1 + O(\frac{\log n}{\sqrt{n}}) \right)\left( \frac{1}{f(t)}\pm \frac{g_1(t)}{f^2(t)} + O(n^{-2/7})\right) \\
  &\subs r(t) \pm \left( \frac{2r(t)}{f(t)} + 1 \right)g_1(t) +O(n^{-2/7}) \subs r(t) \pm \frac{K}{20} g_1(t).
\end{align*}
Similarly, the left-hand side of the second inequality is contained in
\begin{align*}
  &~~~ (r(t)\pm g_1(t) )( y(t)\pm g_2(t) ) \left( 1 + O(\frac{\log^2 n}{\sqrt{n}}) \right)\left( \frac{1}{f(t)}\pm \frac{g_1(t)}{f^2(t)} + O(n^{-2/7})\right) \\
  &\subs r(t)d(t) \pm \left( \frac{r(t)}{f(t)} + \frac{y(t)}{f(t)(1+d(t))} + \frac{r(t)y(t)}{f^2(t)(1+d(t))} \right)g_2(t) +O(n^{-2/7}) \subs r(t)d(t) \pm \frac{K}{20} g_2(t)
\end{align*}
using $y(t)=f(t)d(t)$ and $g_2(t)=(1+d(t))g_1(t)$. Finally, the left-hand side of the last inequality is contained in
\begin{align*}
  &~~~ ( z(t)\pm g_2(t) ) \left( 1 + O(\frac{\log n}{\sqrt{n}}) \right)\left( \frac{1}{f(t)}\pm \frac{g_1(t)}{f^2(t)} + O(n^{-2/7})\right) \\
  &\subs f(t) \pm \left( \frac{1}{f(t)} + \frac{z(t)}{f^2(t)(1+d(t))} \right)g_2(t) +O(n^{-2/7}) \subs f(t) \pm \frac{K}{20} g_1(t)
\end{align*}
using $z(t)=f^2(t)$.
\end{proof}

The remaining lemmas connect the concentration of the original variables and those shifted by the expectations. We will use the following observations in the calculations.

\begin{CLAIM} \label{lem:intsum}
Let $s(t)$ be a differentiable function on $[0,T]$ such that $\sup_{t\in [0,T]} |s'(t)| = O(\polylog n)$ and $t_i=\frac{i}{n^2}$. Then
\[ \frac{1}{n^2}\sum_{i=0}^{j-1} s(t_i) = \int_0^{t_j} s(\tau)d\tau +O(n^{-1}). \]
\end{CLAIM}
\begin{proof}
It is a well-known fact in numerical analysis that for reals $a\le q \le b$
\[ \left| \int_a^b s(\tau)d\tau - (b-a)s(q) \right| \le (b-a)^2\sup_{t\in [a,b]} |s'(t)|.\]
Taking $a=q=t_i$ with $b=t_{i+1}$ and using $t_{i+1}-t_i=\frac{1}{n^2}$, this gives 
\[\left| \int_{t_i}^{t_{i+1}} s(\tau)d\tau - \frac{s(t_i)}{n^2} \right| \le \frac{t_{i+1}-t_i}{n^2}\sup_{t\in [t_i,t_{i+1}]} |s'(t)|,\]
and summing these up for $i=0,\ldots,j-1$, we get 
\[ \left| \int_0^{t_j} s(\tau)d\tau - \frac{1}{n^2}\sum_{i=0}^{j-1} s(t_i)  \right| \le \frac{t_j \cdot \sup_{t\in [0,t_j]} |s'(t)|}{n^2} = O \left(\frac{\polylog n}{n^2}\right) \le O(n^{-1}). \]
\end{proof}

This claim will be applied when $s$ is one of the functions $d',f',x',y',z',g_1'$ and $g_2'$, in which case $s'$ is indeed bounded by $O(\polylog n)$ in the interval $[0,\sqrt{\log n}]$.

\medskip
\begin{CLAIM} \label{lem:intexp}
For $\alpha\in\{1,2\}$ we have
\[ \int_0^{t} g_{\alpha}(\tau) d\tau \le \frac{1}{K} (g_{\alpha}(t)-n^{-1/6}). \]
\end{CLAIM}
\begin{proof}
Note that $g_{\alpha}(t)=\varphi(t)e^{Kt}n^{-1/6}$, where $\varphi(t)$ is either constant 1 or $1+d(t)$. In both cases, $\varphi(0)=1$ and $\varphi'(t)\ge 0$ for $t\ge 0$, so
\[ \frac{g_{\alpha}'(t)}{K} = \left(\frac{1}{K}\varphi(t)e^{Kt} n^{-1/6}\right)' = (\varphi'(t)e^{Kt}/K + \varphi(t)e^{Kt})n^{-1/6} \ge g_{\alpha}(t). \]
Hence
\[ \int_0^{t} g_{\alpha}(\tau) d\tau \le \int_0^{t} \frac{g_{\alpha}'(\tau)}{K} d\tau = \frac{1}{K} (g_{\alpha}(t)-n^{-1/6}), \]
as required.
\end{proof}

It is time to formally define the shifted variables. Recall that if $R(i)$ represents one of our random variables, then we use the non-negative variables $A_i$ and $C_i$ for the one-step increase and decrease in $R$, respectively, so that $R(i)-R(i-1)=A_i-C_i$. Our aim is to show that $R(i)$ is approximately $n^{\gamma}r(t_i)$ for some real $\gamma$, where the error (the allowed fluctuation of $R$) is bounded by $n^{\gamma}g_{\alpha}(t_i)$ for some $\alpha\in \{1,2\}$. Here our choice of $\gamma$ and $\alpha$ depends on the variable $R$ represents: $\gamma$ will be 1 for $F$ and $Z$, $1/2$ for $D$ and $Y$, and 0 for $X$, while $\alpha$ will be 1 for $D$ and $F$, and 2 for $X$, $Y$ and $Z$.

To show the concentration of $R$, we approximate $A_i$ and $C_i$ by their expectations, which, as we shall prove, lie in the intervals $n^{\gamma-2}(r_A(t_{i-1})\pm \frac{K}{2}g_{\alpha}(t_{i-1}))$ and $n^{\gamma-2}(r_C(t_{i-1})\pm \frac{K}{2}g_{\alpha}(t_{i-1}))$, respectively, for some appropriately chosen functions $r_A(t)$ and $r_C(t)$. Thus we can define the shifted variables $A_i^+$ and $C_i^+$ having non-negative expectation, as well as $A_i^-$ and $C_i^-$ having non-positive expectation as follows: 
\begin{align*}
  A_i^{\pm} &=A_i-n^{\gamma-2}(r_A(t_{i-1})\mp \frac{K}{2}g_{\alpha}(t_{i-1})) \quad\quad \mbox{with} \quad\quad B_j^{\pm}=\sum_{i=1}^j A_i^{\pm} \quad \mbox{and} \\
  C_i^{\pm} &=C_i-n^{\gamma-2}(r_C(t_{i-1})\mp \frac{K}{2}g_{\alpha}(t_{i-1})) \quad\quad \mbox{with} \quad\quad D_j^{\pm}=\sum_{i=1}^j C_i^{\pm}.
\end{align*}

\begin{LEMMA} \label{lem:bound}
Suppose the variable $R$ satisfies $R(j)=n^{\gamma}r(0)+\sum_{i=1}^j A_i-C_i$, where $r$ is a polynomial in $t$ such that $r'(t)=r_A(t)-r_C(t)$. Then
\begin{align*}
  R(j) &\le n^{\gamma}(r(t_j)+g_{\alpha}(t_j)) - n^{\gamma-1/6}/2 + B_j^- - D_j^+ \quad \mbox{and} \\
  R(j) &\ge n^{\gamma}(r(t_j)-g_{\alpha}(t_j)) + n^{\gamma-1/6}/2 + B_j^+ - D_j^-.
\end{align*}
\end{LEMMA}
\begin{proof}
Let us first consider the upper bound:
\begin{align*}
  R(j) &= n^{\gamma}r(0)+\sum_{i=1}^j A_i-C_i \\
  &=  n^{\gamma}r(0)+ \sum_{i=1}^j \left(A_i^- + n^{\gamma-2}(r_A(t_{i-1})+ \frac{K}{2}g_{\alpha}(t_{i-1})) \right) - \sum_{i=1}^j \left(C_i^+ + n^{\gamma-2}(r_C(t_{i-1})- \frac{K}{2}g_{\alpha}(t_{i-1})) \right) \\
  &= B_j^- - D_j^+ + n^{\gamma}r(0) + n^{\gamma-2}\sum_{i=1}^j (r_A(t_{i-1})-r_C(t_{i-1})) + n^{\gamma-2}\sum_{i=1}^j Kg_{\alpha}(t_{i-1})
\end{align*}
Now we apply Claim~\ref{lem:intsum} with functions $r_A(t)-r_C(t)$ (a polynomial) and $Kg_{\alpha}(t)$ (a product of a polynomial and an exponential function). As $T\le \sqrt{\log n}$, their derivatives are clearly bounded by $O(\polylog n)$ on $[0,T]$.
\begin{align*}
  R(j) &\le B_j^- - D_j^+ + n^{\gamma}r(0) + n^{\gamma} \int_0^{t_j} (r_A(\tau)-r_C(\tau)) d\tau + n^{\gamma} \int_0^{t_j} Kg_{\alpha}(\tau)d\tau +O(n^{-1})\\
  &\le n^{\gamma}(r(t_j)+g_{\alpha}(t_j))-n^{\gamma-1/6}/2 + B_j^- - D_j^+
\end{align*}
using Claim~\ref{lem:intexp} and $n^{\gamma-1/6}+O(n^{-1}) \ge n^{\gamma-1/6}/2$ in the last step.

The lower bound comes from an analogous argument by changing the appropriate signs.
\end{proof}

Using this, we can estimate the probability that $R(j)$ deviates from its expectation:

\begin{COR} \label{cor:prob}
Suppose the numbers $\gamma, \alpha$ and the functions $R, r, r_A, r_C$ satisfy the conditions of Lemma~\ref{lem:bound}. Suppose furthermore that $B_j^{\pm}$ and $D_j^{\pm}$ are $(\eta_1,N_1)$-bounded and $(\eta_2,N_2)$-bounded martingale pairs, respectively, where $\eta_{\beta} N_{\beta}\le \eps$ and $\eta_{\beta}<N_{\beta}/10$ for $\beta=1,2$. Then the probability that $R(j)\not\in n^{\gamma}(r(t_j)\pm g_{\alpha}(t_j))$ is at most $4e^{-\frac{n^{2\gamma-1/3}}{50\eps j}}$.
\end{COR}
\begin{proof}
Lemma~\ref{lem:bound} shows that $R(j)> n^{\gamma}(r(t_j)+g_{\alpha}(t_j))$ implies $n^{\gamma-1/6}/2 < B_j^- -D_j^+$, hence this event is contained in the union of the events $n^{\gamma-1/6}/4 < B_j^-$ and $-n^{\gamma-1/6}/4 > D_j^+$. A straightforward application of Lemma~\ref{lem:mart} then gives a bound of $e^{-\frac{n^{2\gamma-1/3}}{50\eps j}}$ on the probability of each event, thus $R(j)> n^{\gamma}(r(t_j)+g_{\alpha}(t_j))$ occurs with probability at most $2e^{-\frac{n^{2\gamma-1/3}}{50\eps j}}$. A similar argument using the other inequality of Lemma~\ref{lem:bound} gives the same bound on the probability of the event $R(j)< n^{\gamma}(r(t_j)-g_{\alpha}(t_j))$, finishing the proof.
\end{proof}

\subsection{Degrees} 

Recall that in this section, and also in the next four sections, we assume $\mG_{j-1}$ holds, i.e., the values of $D_v,F_v,X_{u,v},Y_{u,v}$ and $Z_{u,v}$ are all in the prescribed intervals during the first $j-1$ steps.

\begin{proof}[Proof of Proposition~\ref{prop:main}(a)]

Let $A_i$ be the indicator random variable of the event that an open triple at $v$ was successfully sampled in step $i$. Then $D_v(j)=\sum_{i=1}^j A_i$. The probability that $A_{i+1}=1$ is
\[  \frac{2pF_v(i)}{\sum_{w} F_w(i)} \in \frac{2c(f(t_i)\pm g_1(t_i))}{n^{3/2}(f(t_i)\pm g_1(t_i))}   \subs  \frac{1}{n^{3/2}} \left(2c\pm \frac K2e^{Kt_i}n^{-1/6} \right) \]
using Claim~\ref{lem:squeeze}.

Set
\[ A_i^{\pm}=A_i-\frac1{n^{3/2}} \left(2c\mp \frac K2e^{Kt_{i-1}}n^{-1/6} \right) \quad \mbox{and} \quad B_j^{\pm}=\sum_{i=1}^j A_i^{\pm}, \]
then $B_j^{\pm}$ is a $(\frac {3c}{n^{3/2}},1)$-bounded martingale pair. So if we define $C_i$ and $r_C(t)$ to be 0 for all $i$, then all the conditions of Corollary~\ref{cor:prob} are satisfied with the choice of $r_A(t)=2c$, $r(t)=d(t)$, $\gamma=1/2$, $\alpha=1$ and $\eps=3n^{-3/2}$. Hence the probability that $R(j)=D_v(j)$ is not in $\sqrt{n}(d(t_j)\pm g_1(t_j))$ is less than $4e^{\frac{-n^{1/6}}{\log n}}\le n^{-10}$, using $150j\le 150n^2\sqrt{\log n}\le n^2\log n$.
\end{proof}

\subsection{Open triples} 

\begin{proof}[Proof of Proposition~\ref{prop:main}(b)]

Here we break the one-step change in $F_v(i)$ into two parts: $A_i$ will be the gain in the open triples at $v$ caused by the $i$'th sample and $C_i$ will be the loss, so that we can write $F_v(j)=n-1+ \sum_{i=1}^j A_i-C_i$.

We may lose a particular open triple $uwv$ in two different ways: either if we sample it, or if we successfully sample another open triple with the same missing edge $vu$. There are at most $X_{u,v}\le 50\log n$ candidates for this other triple and a successful sample has probability $p=O(1/\sqrt{n})$, so the linearity of expectation gives
\[  \Exp[C_{i+1}] = \frac{2F_v(i)(1+O(\frac{\log n}{\sqrt{n}}))}{\sum_{w} F_w(i)} \in \frac{2(f(t_i)\pm g_1(t_i))(1+O(\frac{\log n}{\sqrt{n}}))}{n(f(t_i)\pm g_1(t_i))} \subs \frac1n \left(2\pm \frac K2e^{Kt_i}n^{-1/6} \right) \]
using Claim~\ref{lem:squeeze}

Set 
\[ C_i^{\pm}=C_i-\frac1n \left(2\mp \frac K2e^{Kt_{i-1}}n^{-1/6} \right) \quad \mbox{and} \quad D_j^{\pm}=\sum_{i=1}^j C_i^{\pm}, \]
then $D_0^{\pm},D_1^{\pm},\ldots$ is a $(\frac 3n, 50\log n)$-bounded martingale pair, because one sample can only ``break'' the open triples with the same missing edge, and there are at most codegree-many of them.

\medskip
On the other hand, as we have already mentioned in Section~\ref{sec:diff}, there are two ways to obtain new open triples at $v$. The contribution of a new edge $vu$ touching $v$ in step $i+1$ is $D_u(i) - X_{u,v}(i)$ because it creates an open triple at $v$ with any edge of $u$ except if the third edge is already there. Alternatively, a new edge incident to a neighbor $u$ of $v$ creates a new open triple unless it connects to another neighbor of $v$. There are at most $\sum_{u,u'\in D_v(i)} X_{u,u'}(i)$ open triples that could create an edge between two neighbors of $v$, so
\[  \Exp[A_{i+1}] = \frac{2p}{\sum F_w(i)}\left( \sum_{wu\in F_v(i)} D_u(i) - X_{u,v}(i) \right) + \frac{2p}{\sum F_w(i)}\left( \sum_{u\in D_v(i)} F_u(i) - \sum_{u,u'\in D_v(i)} O(X_{u,u'}(i))\right). \]
Note how we abuse our notation to also think of the quantities $D_v(i)$ and $F_v(i)$ as the set they count. So $u\in D_v(i)$ should be understood as a neighbor of $v$ and $wu\in F_v(i)$ refers to an open triple $vwu$. Using Claim~\ref{lem:squeeze} we get
\[  \Exp[A_{i+1}] \subs \frac{4c(d(t_i)\pm g_1(t_i))(f(t_i)\pm g_1(t_i))(1-O(\frac{\log n}{\sqrt{n}}))}{n(f(t_i)\pm g_1(t_i))} \subs \frac1n \left( 4cd(t_i) \pm \frac K2 e^{Kt_i}n^{-1/6} \right). \]

This means that for 
\[ A_i^{\pm}=A_i-\frac1n\left (4cd(t_{i-1})\mp \frac K2 e^{Kt_{i-1}}n^{-1/6}\right) \quad \mbox{and} \quad B_j^{\pm} =\sum_{i=1}^j A_{i}^{\pm}, \]
$B_0^{\pm},B_1^{\pm},\ldots$ is a martingale pair.

Next, we show that it is a $(\frac{\log n}{n},\sqrt{n}\log n)$-bounded martingale pair. Indeed, adding an edge $vw$ in step $i+1$ can increase the number of open triples at $v$ by at most $A_i\le D_w(i)$ whereas an edge $ww'$ not touching $v$ can only increase it by one. The upper bound then comes from $D_w(i)=O(\sqrt{n\log n}) \le \sqrt{n}\log n$ and $A_i\ge A_i^{\pm}$. On the other hand, $A_i^{\pm}$ is smallest when $A_i=0$. Observing that $4cd(t)\le 8c^2\sqrt{\log n}$ we see that the change is bounded from below by $(-\log n/n)$.

Therefore we can apply Corollary~\ref{cor:prob} with $r_A(t)=4cd(t)$, $r_C(t)=2$, $r(t)=f(t)$, $\gamma=1$, $\alpha=1$, and $\eps=\log^2 n/\sqrt{n}$ to show that the probability that $R(j)=F_v(j)+1$ (or $F_v(j)$) is not in the interval $n\big(f(t_j)\pm g_1(t_j)\big)$ is at most $4e^{-n^{1/6}/\log^3 n}\le n^{-10}$.
\end{proof}

\subsection{3-walks} 

\begin{proof}[Proof of Proposition~\ref{prop:main}(d)]

Once again, we break the one-step change in $Y_{u,v}(i)$ into two parts: $A_i$ will be the gain in the open 3-walks from $u$ to $v$ caused by the $i$'th sample and $C_i$ will be the loss, so we can write $Y_{u,v}(j)= \sum_{i=1}^j A_i-C_i$.

We lose a particular 3-walk $uww'v$ either if we sample its open triple $uww'$, or if we add the missing edge $uw'$ by successfully sampling some other triple (as before, the latter event is unlikely since the codegrees are small). Then the linearity of expectation and Claim~\ref{lem:squeeze} gives
\[ \Exp[C_{i+1}]= \frac{2Y_{u,v}(i)(1+O(\frac{\log n}{\sqrt{n}}))}{\sum_w F_w(i)} \in \frac{2(y(t)\pm g_2(t))(1+O(\frac{\log n}{\sqrt{n}}))}{n^{3/2}(f(t)\pm g_1(t))} \subs \frac{1}{n^{3/2}}\left(2d(t) \pm \frac K2 g_2(t)\right). \]
So defining
\[ C_i^{\pm}=C_i-\frac{1}{n^{3/2}}\left(2d(t_{i-1}) \mp \frac K2 g_2(t_{i-1})\right) \quad \mbox{and} \quad D_j^{\pm}=\sum_{i=1}^j C_i^{\pm}, \]
we get that $D_0^{\pm},D_1^{\pm},\ldots$ is a $(\frac {\log n}{n^{3/2}},50\log n)$-bounded martingale pair.

\medskip
Now let us look at $A_{i+1}$, the number of ways a new open 3-walk $uww'v$ can be created in step $i+1$. We follow the analysis described in Section~\ref{sec:diff}. If $uw$ is the new edge, then we need to count the 4-walks $utww'v$ in $G_i$ where $w'$ is not $u$ or a neighbor of $u$, and $utw$ is open. Let $N$ be the set of such candidates for $w'$, then $|N|=D_v(i)-O(\log n)$, and the expected contribution to $A_{i+1}$ of this type is
\[  \frac{2p\sum_{w'\in N} Y_{u,w'}(i)}{\sum_r F_r(i)} \in \frac{2c\Big(d(t)\pm g_1(t) +O(\frac{\log n}{\sqrt{n}})\Big)\Big(y(t)\pm g_2(t)\Big)}{n^{3/2}(f(t)\pm g_1(t))} \subs \frac{1}{n^{3/2}}\left( \frac{2cd(t)y(t)}{f(t)} \pm \frac{K}{6} g_2(t) \right)   \]
Strictly speaking, we are using the linearity of expectation over the indicator variables for each fixed 3-walk $uww'v$. The probability that this walk is created is the number of $t$'s such that $utww'$ is an open 3-walk in $G_i$, divided by the number of open triples.

We similarly get that the expected contribution where $ww'$ is the new edge is
\[ \frac{2p\sum_{w'\in N} Y_{w',u}(i)}{\sum_r F_r(i)} \in \frac{1}{n^{3/2}}\left( \frac{2cd(t)y(t)}{f(t)} \pm \frac{K}{6} g_2(t) \right), \]
whereas new open 3-walks where $w'v$ is the new edge come from open 4-walks $uww'tv$ in $G_i$, so the expected contribution of this type is
\[ \frac{2pZ_{u,v}(i)}{\sum_r F_r(i)} \in \frac{2c(z(t)\pm g_2(t))}{n^{3/2}(f(t)\pm g_1(t)} \subs   \frac{1}{n^{3/2}}\left( \frac{2cz(t)}{f(t)} \pm \frac{K}{6} g_2(t)\right). \]

Putting all of these together, we see that for
\[A_i^{\pm} = A_i - \frac{1}{n^{3/2}}\left( \frac{2c\big(2d(t_{i-1})y(t_{i-1})+z(t_{i-1})\big)}{f(t_{i-1})} \pm \frac{K}{2}g_2(t_{i-1}) \right), \]
$B_j^{\pm}=\sum_{i=1}^j A_i^{\pm}$ is a martingale pair. In fact it is $(\frac{\log^2 n}{n^{3/2}},50\log n)$-bounded, since a new edge can contribute at most codegree-many new 3-walks.

Now we can apply Corollary~\ref{cor:prob} with $r_A(t)=2c\big(2d(t)y(t)+z(t)\big)/f(t)$, $r_C(t)=2y(t)/f(t)$, $r(t)=y(t)$ (recall the differential equation that $y$ satisfies to see that $r'=r_A-r_C$), $\gamma=1/2$, $\alpha=2$ and $\eps=\log^4 n/n^{3/2}$ to show that the probability that $R(j)=Y_{u,v}(j)$ is not in the interval $\sqrt{n}\big(y(t_j)\pm g_2(t_j) \big)$ is at most $4e^{-n^{1/6}/\log^5 n}\le n^{-10}$.
\end{proof}

\subsection{4-walks} 

\begin{proof}[Proof of Proposition~\ref{prop:main}(e)]

This time we define $A_i$ to be the number of new open 4-walks created in step $i$ and $C_i$ to be the number of open 4-walks we lose in step $i$, so that $Z_{u,v}(j)=n-2 +\sum_{i=1}^j A_i-C_i$.

Once again, we lose an open 4-walk $uww'w''v$ if one of its open triples $uww'$ or $w'w''v$ is sampled, or if one of their missing edges $uw'$ or $w'v$ is added through a successful sample of a different open triple. Hence we get, using Claim~\ref{lem:squeeze}
\[ \Exp[C_{i+1}]= \frac{4Z_{u,v}(i)(1+O(\frac{\log n}{\sqrt{n}}))}{\sum_w F_w(i)} \in \frac{4(z(t)\pm g_2(t))(1+O(\frac{\log n}{\sqrt{n}}))}{n(f(t)\pm g_1(t))} \subs \frac{1}{n}\left(\frac{4z(t)}{f(t)} \pm \frac K2 g_2(t)\right). \]
So defining
\[ C_i^{\pm}=C_i-\frac{1}{n}\left(\frac{4z(t_{i-1})}{f(t_{i-1})} \mp \frac K2 g_2(t_{i-1})\right) \quad \mbox{and} \quad D_j^{\pm}=\sum_{i=1}^j C_i^{\pm}, \]
we get that $D_0^{\pm},D_1^{\pm},\ldots$ is a $(\frac {\log^2 n}{n},2500\log^2 n)$-bounded martingale pair, because an added edge of the form $uw'$ or $w'v$ can ruin at most $X_{u,w'}\cdot X_{w',v}\le (50\log n)^2$ open 4-paths.

On the other hand, the analysis in Section~\ref{sec:diff} shows that a new open 4-walk $uww'w''v$ can be created in four different ways, based on which one of the four edges was added in step $i+1$. In the case when $uw$ is the new edge, we need to count the 5-walks $utww'w''v$ where $utw$ and $vw''w'$ are open, and $w'$ is not $u$ or a neighbor of $u$. Let $M$ be the set of such edges $w'w''$ for fixed $u$ and $v$. Then 
\[ |M|=F_{v}(i)- Y_{v,u}(i)-O(X_{v,u}(i))=F_{v}(i)-O(\sqrt{n\log^3 n}), \]
hence the expected contribution in this case is
\[  \frac{2p\sum_{w'w''\in M} Y_{u,w'}(i)}{\sum_r F_r(i)} \in \frac{2c\Big(f(t)\pm g_1(t) +O(\frac{\log^2 n}{\sqrt{n}})\Big)\Big(y(t)\pm g_2(t)\Big)}{n(f(t)\pm g_1(t))} \subs \frac 1n \left( \frac{2cf(t)y(t)}{f(t)} \pm \frac{K}{8} g_2(t) \right).   \]

But the remaining three cases are essentially the same, we only need to switch $u$ and $v$ or the two indices of the variables $Y$. This means that
\[ \Exp[A_{i+1}] \in \frac 1n \left( \frac{8cf(t)y(t)}{f(t)} \pm \frac{K}{2} g_2(t) \right), \]
so we can define
\[A_i^{\pm} = A_i - \frac 1n \left( \frac{8cf(t_{i-1})y(t_{i-1})}{f(t_{i-1})} \mp \frac{K}{2} g_2(t_{i-1}) \right) \quad \mbox{and} \quad B_j^{\pm}=\sum_{i=1}^j A_i^{\pm}, \]
where $B_0^{\pm},B_1^{\pm},\ldots$ is a $(\frac{\log^3 n}{n},3\sqrt{n}\log^2 n)$-bounded martingale pair. This is because a new edge of the form $uw$ can add at most $D_v(i)X_{w,w''}(i)= O(\sqrt{n}\log^{3/2} n) \le \sqrt{n}\log^2 n$ new 4-walks and the same bound works for an edge touching $v$, whereas a new edge not touching $u$ and $v$ creates at most $100\log n$ open 4-walks: at most codegree-many in both of the positions $ww'$ and $w'w''$.

Now we apply Corollary~\ref{cor:prob} with $r_A(t)=8cf(t)y(t)/f(t)$, $r_C(t)=4z(t)/f(t)$, $r(t)=z(t)$ (the differential equation for $z$ implies $r'=r_A-r_C$), $\gamma=1$, $\alpha=2$ and $\eps=\log^6 n/\sqrt{n}$ to show that the probability that $R(j)=n+Z_{u,v}(j)$ is not in the interval $n\big(z(t_j)\pm g_2(t_j) \big)$ is at most $4e^{-n^{1/6}/\log^7 n}\le n^{-10}$.
\end{proof}

\subsection{Codegrees} 

\begin{proof}[Proof of Proposition~\ref{prop:main}(c)]

Let $A_i=X_{u,v}(i)-X_{u,v}(i-1)$ be the increase in the codegree of $u$ and $v$ in a step so that $X_{u,v}(j)=1+\sum_{i=1}^j A_i$. It is easy to see that $A_{i+1}$ is the indicator random variable of the event that the open triple of an open 3-walk from $u$ to $v$ or from $v$ to $u$ is successfully sampled in step $i+1$. The probability of this event is
\[  \frac{2p(Y_{u,v}(i)+Y_{v,u}(i))}{\sum_{w} F_w(i)} \in \frac{4c(y(t)\pm g_2(t))}{n^2 (f(t)\pm g_1(t))} \subs \frac{1}{n^2}\left(\frac{4cy(t)}{f(t)} \pm \frac{K}{2}g_2(t) \right). \]

So if we set $A_i^-=A_i-\frac{1}{n^2}\left(\frac{4cy(t_{i-1})}{f(t_{i-1})} + \frac{K}{2}g_2(t_{i-1}) \right)$ then $B_j^-=\sum_{i=1}^j A_i^-$ is a supermartingale and it is $(\frac{10\sqrt{\log n}}{n^2},1)$-bounded. Now we can apply Lemma~\ref{lem:bound} with $\gamma=0$, $\alpha=2$, $r_A(t)=\frac{4cy(t)}{f(t)}=4cd(t)=8c^2t$, $r_C(t)=0$ and $r(t)=4c^2t^2+1$ to $R(j)=X_{u,v}(j)$. Then the first inequality gives
\[ X_{u,v}(j) \le 1+ 4c^2t_j^2+ g_2(t_j) + B_j^-. \]

Therefore (keeping in mind that $t_j\le \sqrt{\log n}$ and $c\le 1$) we see that if $X_{u,v}(j)>50\log n$ then $B_j^- > 25\log n$. But by Lemma~\ref{lem:mart} this has probability at most
\[ e^{-\frac{25^2\log^2 n}{30 \log n}} \le e^{-10\log n} =n^{-10}\]
for any $j\le n^2\sqrt{\log n}$, finishing our claim.
\end{proof}

\section{The second phase}  \label{sec:phasetwo}

In this section we analyze the second phase of the process and prove our main result, the lower and upper bounds on the threshold probability. Unlike in the first phase, where we made one step at a time, here we expose triples in rounds. In a round we simultaneously sample all the currently open triples, and then add the edges accordingly.

Let us adapt our notation to the second phase as follows. From now on $D_v(i),~i=0,1,\ldots$ will denote the degree of the vertex $v$ after $i$ rounds in the second phase. For example, $D_v(0)$ is the degree of $v$ at the end of the first phase, i.e., $D_v(Tn^2)$ with the old notation. We similarly re-define the other variables $F_v,X_{u,v},Y_{u,v}$ and $Z_{u,v}$, and let $G_i$ denote the graph after the $i$'th round.

\medskip
We will make use of the following Chernoff-type inequalities (see, e.g., \cite{JLRBOOK}).
\begin{CLAIM} \label{lem:chern}
Let $X\sim \Bin (n,p)$ be a binomial random variable. Then
\begin{enumerate}
  \item $\Prb[X > np+a] \le e^{-\frac{a^2}{2(np+a/3)}}$ and
  \item $\Prb[X < np-a] \le e^{-\frac{a^2}{2np}}.$
\end{enumerate}
\end{CLAIM}

\subsection{The lower bound}

Suppose $c<\frac{1}{2}$ is some fixed constant. Before we start the second phase, we need to decide how many steps the first phase should take. Recall that $f(t)$ has a root at $T_0= \frac{1-\sqrt{1-4c^2}}{4c^2}$ and that it is monotone decreasing in the interval $[0,T_0]$. It is easy to check that $d(T_0)<1$, so fix a positive constant $\delta < 1-d(T_0)$ and choose $\eps>0$ so that $\frac{c\eps}{1-2c} < \delta$. We define the stopping time $T$ to be in the interval $[0,T_0]$ so that $f(T)=\eps/2$. Hence if we apply Theorem \ref{thm:diffeq} with this $T$, we get that after $Tn^2$ steps
\begin{itemize}
\item $D_v(0) \le (d(T)+g_1(T))\sqrt{n} \le (1-\delta)\sqrt{n}$ and
\item $F_v(0) \le (\eps/2+g_1(T))n\le \eps n$
\end{itemize}
for every vertex $v$. At this point, we move on to the second phase of the process.

The plan is to show that the second phase ends in $O(\log n)$ rounds, while all the degrees stay below $\sqrt{n}$. This would imply that the final graph has at most $n\sqrt{n}$ edges, in particular, it is not complete. The following statement bounds the degrees of the vertices in the first $O(\log n)$ rounds. Showing that in the meantime the second phase gets stuck will be an easy corollary.

\begin{CLAIM} \label{lem:lower}
Let $m=4\log_{1/2c}{n}$. Then, with high probability, $D_v(i) < \sqrt{n}$  for every vertex $v$ and $0\le i \le m$.
\end{CLAIM}
\begin{proof}
We will prove by induction that with high probability
\begin{itemize}
\item $D_v(i) \le \left(1-\delta + (1+2c+\ldots+(2c)^{i-1})c\eps+i\cdot n^{-1/6}\right)\sqrt{n}$ \quad and
\item $F_v(i) \le ((2c)^i\eps+2n^{-1/6})n$
\end{itemize}
hold for every vertex $v$ and $1\le i\le m$. Note that, by our choice of $\eps$, the bound on the degrees is less than $\left(1-\delta + \frac{c\eps}{1-2c}+i\cdot n^{-1/6}\right)\sqrt{n} <\sqrt{n}$.

To proceed with the induction, we condition on the event that the bounds hold for $i$ and then estimate the probability that they fail for $i+1$ for some vertex $v$.

First we show that the degree of each vertex increases by at most $\left((2c)^ic\eps+n^{-1/6}\right)\sqrt{n}$ in round $i+1$. Indeed, the number of new edges that touch the vertex $v$ is stochastically dominated by the binomial distribution $\Bin\left(F_v(i),\frac{c}{\sqrt{n}}\right)$. Hence, by the first Chernoff-bound in Claim~\ref{lem:chern},
\[  \Prb\left[ D_v(i+1)-D_v(i) >  F_v(i)\frac{c}{\sqrt{n}} + (1-2c)n^{1/3} \right] < e^{-\Omega(n^{1/6})},  \]
so a union bound over all the vertices shows that the first bound fails in round $i+1$ with probability at most $e^{-\Omega(n^{1/7})}$.

The second inequality follows from the first one by an easy counting argument. Since we sample all the current open triples every round, the ones counted in $F_v(i+1)$ are all new triples, i.e., they contain at least one new edge added in round $i+1$. Now an open triple either has a new edge incident to $v$ or not. If it does, we can choose it in at most $\left((2c)^ic\eps+n^{-1/6}\right)\sqrt{n}$ ways, and then extend each choice in at most $\sqrt{n}$ ways to get a triple (as all degrees are below $\sqrt{n}$). If not, then we first choose a neighbor of $v$ and then a new incident edge. Consequently, the total number of open triples at $v$ is
\[ F_v(i+1)\le 2\cdot\left((2c)^ic\eps+n^{-1/6}\right)\sqrt{n}\cdot \sqrt{n}=((2c)^{i+1}\eps+2n^{-1/6})n. \]
Taking a union bound over all the $m$ rounds then completes the proof.
\end{proof}

\begin{COR} \label{cor:lower}
Let $Q(i)$ be the total number of open triples after $i$ rounds. Then $Q(m) = 0$ with high probability.
\end{COR}
\begin{proof}
If a triple is open after the $i$'th round, then it contains at least one new edge. Of course, the number of open triples containing some fixed new edge $uv$ is at most $D_v(i)+D_u(i)\le 2\sqrt{n}$, whp. On the other hand, the number of new edges cannot exceed the number of positive samples in the $i$'th round, distributed as $\Bin(Q(i-1),\frac{c}{\sqrt{n}})$. Putting these together, this means that $Q(0),\ldots,Q(m)$ is a sequence of random variables where $Q(i)$ is stochastically dominated by $2\sqrt{n}\cdot\Bin(Q(i-1),\frac{c}{\sqrt{n}})$. In particular,
\[
\Exp[Q(i)]=\Exp[\Exp[Q(i)|Q(i-1)]] \le \Exp[2c Q(i-1)]=2c\Exp[Q(i-1)].
\]
Using $Q(0)\le n^3$, a simple application of Markov's inequality gives
\[ \Prb[Q(m)>0]\le \Exp[Q(m)]\le (2c)^m Q(0)\le n^{-4}\cdot n^3 =o(1). \]
\end{proof}

\begin{proof}[Proof of Theorem~\ref{thm:main}, part 2]

Corollary~\ref{cor:lower} shows that whp the process runs out of open triples after at most $m$ rounds in the second phase. According to Claim~\ref{lem:lower}, at this final stage all vertices have degree at most $\sqrt{n}$, i.e., the graph has at most $\frac{n^{3/2}}{2}$ edges whp.
\end{proof}

\subsection{The upper bound}
Suppose $\frac 12<c\le 1$ is fixed. Then we can run the first phase all the way, for $n^2\sqrt{\log n}$ steps. Indeed, as the function $f(t)$ has a global minimum of $f\left(\frac{1}{4c^2}\right) = 1-\frac{1}{4c^2}>0$, we can apply Theorem \ref{thm:diffeq} with stopping time $T=\sqrt{\log n}$.

Our plan is to give rapidly increasing lower bounds on the degrees and codegrees as the graph evolves, thus showing that we reach the complete graph in $O(\log\log n)$ rounds. Let us analyze the first round separately.

The initial parameters of the second phase are, as implied by Theorem \ref{thm:diffeq},
\begin{itemize}
\item $X_{u,v}(0) \le 50\log n,$
\item $Z_{u,v}(0) = 16c^4n\log^2 n + O(n \log^{3/2} n) \ge 2 c^4n\log^2 n.$
\end{itemize} 
for any vertices $u$ and $v$.

\begin{LEMMA} \label{lem:stepone}
There is some constant $\gamma>0$, such that the codegree $X_{u,v}(1)\ge \gamma\log^2n$ for every pair of vertices $u,v$ with high probability.
\end{LEMMA}
\begin{proof}
Fix $u$ and $v$. We expect most of their new common neighbors to be vertices $w$ with open triples to both $u$ and $v$. So if $\tilde X_{p,r}$ denotes the number of open triples $pqr$, then we want many vertices $w$ such that both $\tilde X_{u,w}$ and $\tilde X_{w,v}$ are relatively large.

\begin{CLAIM}
For every pair of vertices $u,v$, there are at least $a\cdot n$ vertices $w$ such that $\tilde X_{u,w}(0),\tilde X_{v,w}(0) \ge b\log n$, where $a=\frac{c^4}{2500}$ and $b=\frac{c^4}{50}$ are positive constants.
\end{CLAIM}
\begin{proof}
Note that an open 4-walk is just a sequence of two open triples, hence
\[
2c^4n\log^2n \le Z_{u,v}(0) = \sum_{w\in V\setminus\{u,v\}} \tilde X_{u,w}(0)\cdot\tilde X_{v,w}(0).
\]
Here each summand is bounded by $(50\log n)^2$, so if fewer than $an$ vertices $w$ satisfy $c^4\log^2 n \le \tilde X_{u,w}(0)\cdot \tilde X_{v,w}(0)$, then the 
right hand sum above is less than  $c^4\log^2 n \cdot n +(50\log n)^2 \cdot an =2c^4 n \log^2 n$, a contradiction.
At the same time, the bound on the codegrees implies that each $w$ with $c^4\log^2 n \le \tilde X_{u,w}(0)\cdot \tilde X_{v,w}(0)$ satisfies our requirements.
\end{proof}

Now if some $w$ shares at least $b\log n$ open triples with both $u$ and $v$, then it becomes a new common neighbor of them after the first round with probability at least $\left(1-\left(1-\frac{c}{\sqrt{n}}\right)^{b\log n}\right)^2 \ge \left(\frac{cb\log n}{2\sqrt{n}}\right)^2$ (here and later in this section we use that 
$(1-\alpha)^{\beta} \le 1-\alpha \beta/2$ for all $\alpha \beta \le 1$). These events are independent for different $w$'s, hence $X_{u,v}$ is bounded from below by the Binomial random variable $\mbox{Bin}\left(an,\frac{c^2b^2\log^2 n}{4n} \right)$. Then by Lemma~\ref{lem:chern}, the probability that $X_{u,v}(1)$ is smaller than $\gamma \log^2 n$ is $e^{-\Omega(\log^2 n)}$ for a sufficiently small $\gamma$. A union bound over all pairs of vertices finishes the proof.
\end{proof}

To make our life easier, we consider a slightly different second phase from this point on. Instead of sampling open triples with success probability $p$, we will consider a sprinkling process, and sample \emph{all} triples with success probability $\frac{4}{\sqrt{n\log n}}$ in each round (starting from round 2). This means that some triples will have a higher than $p$ chance to exist, but as long as the number of rounds $m$ is $O(\log\log n)$, the effect is negligible: each triple is still sampled with probability at most $\frac{c+o(1)}{\sqrt{n}}$. Formally we can say that we are proving the result for any constant $c'>c$.

\medskip
To give a lower bound on the codegrees in $G_{i+1}$, we define the following sequence:
\[ x_i=\gamma^{2^{i-1}}\log^{2^{i-1}+1}n, \quad i=1,\ldots,m \]
with $x_{m+1}=\frac {n}{10}$,
where we choose $m = O(\log\log n)$ to be smallest possible such that $x_m\ge\frac 14\sqrt{n \log n}$. Let us also set $p_i=1-\left(1-\frac{4}{\sqrt{n\log n}}\right)^{x_{i-1}}$.

\begin{LEMMA}\label{lem:boundx}
With high probability 
\[ X_{u,v}(i) \ge  x_i \]
for all $1\le i \le m+1$ and all pairs of vertices $u$ and $v$.
\end{LEMMA}
\begin{proof}
Lemma~\ref{lem:stepone} shows that the lower bound on the codegrees holds for $i=1$, so assume $2\le i$. We condition on the event that the statement holds for $i-1$ and bound the probability that it fails for $i$.

We claim that under these conditions $G(n,p_i)$ is a subgraph of $G_i$. To see this, observe that if an edge $uv$ is missing from $G_{i-1}$, then it has $X_{u,v}(i-1)\ge x_{i-1}$ independent chances of probability $4/\sqrt{n\log n}$ of being added in the $i$'th round. Moreover, these events are independent for the different non-edges, as the triples are sampled independently, and each triple has at most one missing edge. This means that missing edges are added independently with probability at least $p_i$ while existing edges are kept in the graph, thus indeed $G(n,p_i)\subs G_i$.

We intend to use the Chernoff bound to show that all the codegrees in $G(n,p_i)$, and thus also in $G_i$, exceed $x_i$. For this, observe that the codegree of any fixed pair of vertices in $G(n,p_i)$ is a binomial random variable $R_i\sim \Bin(n-2,p_i^2)$. A straightforward calculation gives $\Exp[R_i] > 2x_i$.
 
Indeed, for $2\le i \le m$ we have
\[ (n-2) \left( 1- \left(1-\frac{4}{\sqrt{n\log n}}\right)^{x_{i-1}}\right)^2 \ge \frac{(n-2)\cdot 4x_{i-1}^2 }{n\log n}  >  \frac{2x_{i-1}^2 }{ \log n}=2x_i.
\]
whereas for $i=m+1$ (using $x_m\ge \frac 14\sqrt{n \log n}$),
\[ (n-2) \left( 1- \left(1-\frac{4}{\sqrt{n\log n}}\right)^{x_m}\right)^2 \ge (n-2)(1-1/e)^2> 2x_{m+1},
\]
Thus, as $x_i \ge \delta\log^2 n$, Claim~\ref{lem:chern} shows that $\Prb [R_i < x_i ] = e^{-\Omega( \log^2 n)}$. Now taking the union bound over all vertex pairs and over all $i$ finishes the proof.
\end{proof}

\begin{proof}[Proof of Theorem~\ref{thm:main}, part 1]
We claim that $G_{m+2}$ is the complete graph. Indeed, Lemma~\ref{lem:boundx} shows that whp all the codegrees in $G_{m+1}$ are linear, so the probability that a fixed edge is missing from $G_{m+2}$ is at most $(1-4/\sqrt{n\log n})^{\Omega(n)}=e^{-\Omega(\sqrt{n/\log n})}$. A union bound over all pairs of vertices then completes the proof.
\end{proof}

\section{Concluding remarks} \label{sec:last}

Probably the most natural question that one can ask is the following. What happens if the process starts with some other tree, and not the star? Intuitively it seems that we are in a worse situation as there are fewer open triples to start with. We would therefore expect that if $p\le \frac{1-\eps}{2\sqrt{n}}$, then starting with any fixed tree, the triadic process fails to propagate whp. In fact, we believe that whp this holds for all trees simultaneously.

Using the topology language, this is equivalent to saying that $p=\frac{1}{2\sqrt{n}}$ is the threshold probability for a random 2-complex to contain a collapsible hypertree (the upper bound comes from Corollary~\ref{cor:top}). We must note that a complex can have trivial fundamental group without actually containing a collapsible hypertree. A yet stronger question would be to ask for a lower bound matching the bound on the threshold in  Corollary~\ref{cor:top} for being simply connected.

Going in a different direction, it would also be interesting to study similar processes that are perhaps more meaningful from the social networks point of view. For example, a triadic process where vertices are discouraged to reach high degrees could be a more realistic model.

\bigskip

\noindent{\bf Acknowledgement.}\, This research was done while the second author was a visiting student at ETH Zurich. He would like to thank the Mathematics Department of ETH for the hospitality and for creating a stimulating research environment.

\medskip
\noindent{\bf Note added in proof.}\, After this paper was written, we learned that Gundert and Wagner \cite{GW15} independently improved the bound in \cite{BHK11}, and showed that $Y_2(n,p)$ is whp simply connected for $p=\frac{c}{\sqrt{n}}$ where $c$ is a sufficiently large constant.

\end{document}